\documentclass[12pt]{article}
\usepackage{amsthm}
\usepackage{latexsym}
\usepackage[latin1]{inputenc}
\usepackage{amsfonts}
\usepackage{multirow}
\usepackage{hhline}
\usepackage{graphics}
\usepackage{graphicx}
\usepackage{amsmath}
\usepackage{amsfonts}
\usepackage{amssymb}
\usepackage{harvard}
\expandafter\ifx\csname amssym.def\endcsname\relax \else \fi
\expandafter\edef\csname amssym.def\endcsname{%
       \catcode`\noexpand\@=\the\catcode`\@\space}
\catcode`\@=11
\def\undefine#1{\let#1\undefined}
\def\newsymbol#1#2#3#4#5{\let\next@\relax
 \ifnum#2=\@ne\let\next@\msafam@\else
 \ifnum#2=\tw@\let\next@\msbfam@\fi\fi
 \mathchardef#1="#3\next@#4#5}
\def\mathhexbox@#1#2#3{\relax
 \ifmmode\mathpalette{}{\m@th\mathchar"#1#2#3}%
 \else\leavevmode\hbox{$\m@th\mathchar"#1#2#3$}\fi}
\def\hexnumber@#1{\ifcase#1 0\or 1\or 2\or 3\or 4\or 5\or 6\or 7\or 8\or
 9\or A\or B\or C\or D\or E\or F\fi}

\font\tenmsa=msam10 at 16pt\relax
\font\sevenmsa=msam7 at 12pt\relax
\font\fivemsa=msam5 at 9pt\relax
\newfam\msafam
\textfont\msafam=\tenmsa
\scriptfont\msafam=\sevenmsa
\scriptscriptfont\msafam=\fivemsa
\edef\msafam@{\hexnumber@\msafam}
\mathchardef\dabar@"0\msafam@39
\def\dashrightarrow{\mathrel{\dabar@\dabar@\mathchar"0\msafam@4B}}
\def\dashleftarrow{\mathrel{\mathchar"0\msafam@4C\dabar@\dabar@}}

\def\ulcorner{\delimiter"4\msafam@70\msafam@70 }
\def\urcorner{\delimiter"5\msafam@71\msafam@71 }
\def\llcorner{\delimiter"4\msafam@78\msafam@78 }
\def\lrcorner{\delimiter"5\msafam@79\msafam@79 }
\def\yen{{\mathhexbox@\msafam@55}}
\def\checkmark{{\mathhexbox@\msafam@58}}
\def\circledR{{\mathhexbox@\msafam@72}}
\def\maltese{{\mathhexbox@\msafam@7A}}

\font\tenmsb=msbm10 at 16pt\relax
\font\sevenmsb=msbm7 at 12pt\relax
\font\fivemsb=msbm5 at 9pt\relax
\newfam\msbfam
\textfont\msbfam=\tenmsb
\scriptfont\msbfam=\sevenmsb
\scriptscriptfont\msbfam=\fivemsb
\edef\msbfam@{\hexnumber@\msbfam}
\def\Bbb#1{{\fam\msbfam\relax#1}}
\def\widehat#1{\setbox\z@\hbox{$\m@th#1$}%
 \ifdim\wd\z@>\tw@ em\mathaccent"0\msbfam@5B{#1}%
 \else\mathaccent"0362{#1}\fi}
\def\widetilde#1{\setbox\z@\hbox{$\m@th#1$}%
 \ifdim\wd\z@>\tw@ em\mathaccent"0\msbfam@5D{#1}%
 \else\mathaccent"0365{#1}\fi}
\font\teneufm=eufm10 at 16pt\relax
\font\seveneufm=eufm7 at 12pt\relax
\font\fiveeufm=eufm5 at 9pt\relax
\newfam\eufmfam
\textfont\eufmfam=\teneufm
\scriptfont\eufmfam=\seveneufm
\scriptscriptfont\eufmfam=\fiveeufm

\csname amssym.def\endcsname

\input epsf

\newtheorem{proposition}{Proposition}
\def\E{\operatorname{\bf E}}

\def\1{{\bf 1}}
\newcommand{\R}{I \! \! R}

\def\C{\Bbb C}
\def\Q{\Bbb Q}

\def\beq{\begin{equation}}
\def\eeq{\end{equation}}

\renewcommand{\Re}{\operatorname{\rm Re}}
\def\process#1{#1=\{#1_t\}_{t\geq 0}}
\newtheorem{theorem}{Theorem}[section]
\newtheorem{remark}{Remark}[section]

\newtheorem{lem}{Lemma}[section]
\newtheorem{corollary}{Corollary}[section]

\newtheorem{assumption}{Assumption}
\begin{document}
\title{Skewness Premium with L\'{e}vy Processes\thanks{We thank seminar participants at École Polytechnique, HEC Paris,
Universit\"{a}t de Barcelona, IMPA, PUC-Rio, Universit\"{a}t de
Navarra, Séminaire Groupe Parisien Bachelier. The ususal disclaimer
applies}}
\author{Jos\'{e} Fajardo\thanks{IBMEC Business School, Rio de
Janeiro - Brazil, e-mail: pepe@ibmecrj.br}\hspace{0.2cm}  and
 Ernesto Mordecki\thanks
{Centro de Matem\'atica, Facultad de Ciencias, Universidad de la
Rep\'ublica, Montevideo. Uruguay. e-mail: mordecki@cmat.edu.uy.}}
\date{\today }

\maketitle
\begin{center}
First draft 05/10/06
\end{center}
\begin{abstract}
We study the skewness premium (SK) introduced by
\citeasnoun{Bates91} in a general context using L\'evy Processes.
Under a symmetry condition \citeasnoun{FajardoMordecki2006b}
obtain that SK is given by the Bate's $x\%$ rule. In this paper we
study SK under the absence of that symmetry condition. More
exactly, we derive sufficient conditions for SK to be positive, in
terms of the characteristic triplet of the Lévy Process under the
risk neutral measure. \vspace{3 mm }
\newline {\bf Keywords:} Skewnes Premium; Lévy Processes.
\\
{\bf JEL Classification:} C52; G10
\end{abstract}
\section{Introduction}

The stylized facts of option prices have been studied by many
authors in the literature. An important fact from option prices is
that relative prices of out-of-the-money calls and puts can be
used as a measure of symmetry or skewness of the risk neutral
distribution. \citeasnoun{Bates91}, called this diagnosis
``skewness premium'', henceforth SK. He analyzed the behavior of
SK using three classes of stochastic processes: Constant
Elasticity of Variance (CEV), Stochastic Volatility and
Jump-diffusion. He found conditions on the parameters for the SK
be positive or
negative.\\

But, as many models in the literature have shown, the behavior of
the assets underlying options is very complex, the structure of
jumps observed is more complex than Poisson jumps. They have
higher intensity, see for example \citeasnoun{Yacine2004}. For
that reason diffusion models cannot consider the discontinuous
sudden movements observed on asset prices. In that sense, the use
of more general process as L\'evy processes have shown to provide
a better fit with real data, as was reported in
\citeasnoun{CarrWu04} and \citeasnoun{EberleinKellerPrause98}. On
the other hand,
 the mathematical tools behind these processes are very well
established and known.  \\

When the underlying follows a Geometric Lévy Process,
\citeasnoun{FajardoMordecki2006b} obtained a relationship between
calls and puts, that they called \emph{Put-Call duality} and
obtain as a particular case the \emph{Put-call symmetry}, and
obtain that SK is given by the Bate's $x\%$ rule. The Put-Call duality has important applications, in particularly the Put-Call symmetry, as \citeasnoun{BC94} and \citeasnoun{CEG98} show using symmetry we can construct stactic hedges for exotic options..\\

 In this paper we study the SK under absence of symmetry and
obtain sufficient conditions for the excess of SK be positive or
negative. The main idea behind the proofs is to exploit the
monotonicity property of option prices with respect to some
parameter of the Lévy measure. This monotonicity is not an easy task, monotonicity with respect to the intensity parameter of
the jump have been recently address by \citeasnoun{EkstromTysk05}, while the monotonicity with respect
to the symmetry parameter have not been totally addressed in previous works.
A particular answer is given for the case of GH distributions in
\citeasnoun{Bergenthum07}.\\

 The paper is organized as
follows: in Section 2 we introduce the L\'evy processes and  we
present the duality results. In Section 3 we discuss market
symmetry and present our main results. In Section 4 we study the
skewness premium. Section 5 discuss monotonicity with respect to
the symmetry parameter and Section 6 concludes.

\section{Lévy processes and Duality}
Consider a real valued stochastic process $\process X$, defined on
a stochastic basis ${\cal B}=(\Omega, {\cal F},{\bf F}=({\cal
F}_t)_{t\geq 0}, \Q)$, being c\`adl\`ag, adapted, satisfying
$X_0=0$, and such that for $0\leq s< t$ the random variable
$X_t-X_s$ is independent of the $\sigma$-field ${\cal F}_s$, with
a distribution that only depends on the difference $t-s$. Assume
also that the stochastic basis ${\cal B}$ satisfies the usual
conditions (see \citeasnoun{jacodShiryaev87}). The process $X$ is
a L\'evy process, and is also called a process with stationary
independent increments (PIIS). For general reference on L\'evy
processes see \citeasnoun{jacodShiryaev87},
\citeasnoun{Skorokhod91}, \citeasnoun{Bertoin96},
\citeasnoun{Sato99}. For L\'evy process in Finance see \citeasnoun{BL02}, \citeasnoun{Schoutens2003} and \citeasnoun{CT04}.\\

In order to characterize the law of $X$ under $\Q$, consider, for
$q\in\R$ the L\'evy-Khinchine formula, that states
\begin{equation}\label{e:lk}
\E e^{iq X_t}= \exp\Big\{t\Big[iaq-\frac 12\sigma^2q^2+
\int_{\R}\big(e^{iq y}-1-iq h(y)\big)\Pi(dy)\Big]\Big\},
\end{equation}
with
\begin{equation*}
h(y)=y{\bf 1}_{\{|y|<1\}}
\end{equation*}
a fixed truncation function, $a$ and $\sigma\geq 0$ real
constants, and $\Pi$ a positive measure on ${\R}\setminus\{0\}$
such that $\int (1\wedge y^2)\Pi(dy)<+\infty$, called the
\emph{L\'evy measure}. The triplet $(a,\sigma^2,\Pi)$ is the
\emph{characteristic triplet} of the process, and completely
determines its law.\\

Consider the set
\begin{equation}\label{e:cecero}
{\C_0}=\Big\{z=p+iq\in\C\colon
\int_{\{|y|>1\}}e^{py}\Pi(dy)<\infty\Big\}.
\end{equation}
The set ${\C_0}$ is a vertical strip in the complex plane,
contains the line $z=iq\ (q\in\R)$, and consists of all complex
numbers $z=p+iq$  such that $\E e^{pX_t} <\infty$ for some $t>0$.
Furthermore, if $z\in{\C_0}$, we can define the
\emph{characteristic exponent} of the process $X$, by
\begin{equation}\label{e:char-exp}
\psi(z)=az+\frac 12\sigma^2z^2+
\int_{\R}\big(e^{zy}-1-zh(y)\big)\Pi(dy)
\end{equation}
this function $\psi$ is also called the {\it cumulant} of $X$,
having $\E|e^{zX_t}|<\infty$ for all $t\geq 0$, and $\E
e^{zX_t}=e^{t\psi(z)}$. The finiteness of this expectations
follows from Theorem 21.3 in \citeasnoun{Sato99}. Formula
\eqref{e:char-exp} reduces to formula \eqref{e:lk} when
$\Re(z)=0$.
\subsection{L\'evy market}
By a \emph{L\'evy market} we mean a
model of a financial market with two assets: a deterministic
savings account $\process B$, with
\begin{equation*}
B_t=e^{rt},\qquad r\ge 0,
\end{equation*}
where we take $B_0=1$ for simplicity, and a stock $\process S$,
with random evolution modelled by
\begin{equation}\label{e:market}
S_t=S_0e^{X_t},\qquad S_0=e^x>0,
\end{equation}
where $\process X$ is a L\'evy process.\\

In this model we assume that the stock pays dividends, with
constant rate $\delta\ge 0$, and that the given probability
measure $\Q$ is the chosen equivalent martingale measure. In other
words, prices are computed as expectations with respect to
$\Q$, and the discounted and reinvested process $\{e^{-(r-\delta)t}S_t\}$ is a $\Q$-martingale.\\

In terms of the characteristic exponent of the process this means
that

\begin{equation}\label{e:psiuno}
\psi(1)=r-\delta,
\end{equation}
based on the fact, that $\E
e^{-(r-\delta)t+X_t}=e^{-t(r-\delta+\psi(1))}=1$, and condition
\eqref{e:psiuno} can also be formulated in terms of the
characteristic triplet of the process $X$ as
\begin{equation}\label{e:a}
a=r-\delta-\sigma^2/2-\int_{\R}\big(e^y-1-h(y)\big)\Pi(dy).
\end{equation}
In the case, when
\begin{equation}\label{e:bsm}
X_t=\sigma W_t + at \quad (t\ge 0),
\end{equation}
where $\process W$ is a Wiener process, we obtain the
Black--Scholes--Merton (1973) model (see
\citeasnoun{BlackScholes73},\citeasnoun{Merton73}).\\

In the market model considered we introduce some derivative
assets. More precisely, we consider call and put options, of both
European and American types. Denote by ${\cal M}_T$ the class of
stopping times up to a fixed constant time $T$, i.e:
$${\cal M}_T=\{\tau:\;0\leq \tau \leq T,\;\tau \hbox{ stopping time w.r.t ${\bf F}$}\}.$$
 Then, for each stopping time $\tau\in{\cal M}_T$ we introduce

\begin{align}
c(S_0,K,r,\delta,\tau,\psi)&=\E e^{-r\tau}(S_{\tau}-K)^+,\label{e:call}\\
p(S_0,K,r,\delta,\tau,\psi)&=\E
e^{-r\tau}(K-S_{\tau})^+.\label{e:put}
\end{align}
In our analysis \eqref{e:call} and \eqref{e:put} are auxiliary
quantities, anyhow, they are interesting by themselves as random
maturity options, as considered, for instance, in
\citeasnoun{Schroder99} and \citeasnoun{Detemple01}. If $\tau=T$,
formulas \eqref{e:call} and \eqref{e:put} give the price of the
European call and put options respectively.
\subsection{Put Call duality and dual markets}
\begin{lem}[Duality]\label{p:duality}
Consider a L\'evy market with driving process $X$ with
characteristic exponent $\psi(z)$, defined in \eqref{e:char-exp},
on the set $\C_0$ in \eqref{e:cecero}. Then, for the expectations
introduced in \eqref{e:call} and \eqref{e:put} we have
\begin{equation}\label{e:duality}
c(S_0,K,r,\delta,\tau,\psi)=p(K,S_0,\delta,r,\tau,\tilde{\psi}),
\end{equation}
where
\begin{equation}\label{e:psitilde}
\tilde{\psi}(z)=\tilde{a}z+\frac 12\tilde{\sigma}^2z^2+
\int_{\R}\big(e^{zy}-1-zh(y)\big)\tilde{\Pi}(dy)
\end{equation}
is the characteristic exponent (of a certain L\'evy process) that
satisfies
\begin{equation*}
\tilde{\psi}(z)=\psi(1-z)-\psi(1), \quad \text{for $1-z\in\C_0$},
\end{equation*}
and in consequence,
\begin{equation}\label{e:tripletilde}
\begin{cases}
\tilde{a}        &=\delta-r-\sigma^2/2-\int_{\R}\big(e^y-1-h(y)\big)\tilde{\Pi}(dy),\\
\tilde{\sigma}        &=\sigma,\\
\tilde{\Pi}(dy)&=e^{-y}\Pi(-dy).\\
\end{cases}
\end{equation}
\end{lem}
\begin{proof}
See \citeasnoun{FajardoMordecki2006b}.
\end{proof}

The above Duality Lemma motivates us to introduce the following
market model. Given a L\'evy market with driving process
characterized by $\psi$ in \eqref{e:char-exp}, consider a market
model with two assets, a deterministic savings account
$\process{\tilde B}$, given by
\begin{equation*}
{\tilde B}_t=e^{\delta t},\qquad \delta\ge 0,
\end{equation*}
and a stock $\process{\tilde S}$, modelled by
\begin{equation*}
\tilde{S}_t=Ke^{\tilde{X}_t},\qquad\tilde{S}_0=K>0,
\end{equation*}
where $\process{\tilde{X}}$ is a L\'evy process with
characteristic exponent under $\tilde{\Q}$ given by $\tilde{\psi}$
in \eqref{e:psitilde}. The process $\tilde{S}_t$ represents the
price of $KS_0$ dollars measured in units of stock $S$. This
market is the \emph{auxiliary market} in \citeasnoun{Detemple01},
and we call it \emph{dual market}; accordingly, we call
\emph{Put--Call duality} the relation \eqref{e:duality}. It must
be noticed that \citeasnoun{PeskirShiryaev2001} propose the same
denomination for a different relation. Finally observe, that in
the dual market (i.e. with respect to $\tilde{\Q}$), the process
$\{e^{-(\delta-r)t}{\tilde S}_t\}$ is a martingale. As a
consequence, we obtain the Put--Call symmetry in the
Black--Scholes--Merton model: In this case $\Pi=0$, we have no
jumps, and the characteristic exponents are

\begin{align*}
\psi(z)&=(r-\delta-\sigma^2/2)z+\sigma^2z^2/2,\\
\tilde{\psi}(z)&=(\delta-r-\sigma^2/2)z+\sigma^2z^2/2.
\end{align*}
and relation \eqref{e:duality} is the result known as put--call
symmetry. In the presence of jumps like the jump-diffusion model
of \citeasnoun{Merton76}, if the jump returns of $S$ under $\Q$
and $\tilde{S}$ under $\tilde{\Q}$ have the same distribution, the
Duality Lemma, implies that by exchanging the roles of $\delta$ by
$r$  and $K$ by $S_0$ in (\ref{e:duality}) and
(\ref{e:tripletilde}), we can obtain an American call price
formula from the American put price formula. Motivated by this
analysis we introduce the definition of symmetric markets in the
following section.

\section{Market Symmetry}
It is interesting to note that in a market with no jumps (i.e. in
the Black-Scholes model), the distribution of the discounted and
reinvested stock both in the given risk neutral and in the dual
L\'evy market, taking equal initial values, coincide. It is then
natural to define a L\'evy market to be \emph{symmetric} when this
relation hold, i.e. when

\begin{equation}\label{e:symmetry}
{\cal L}\big(e^{-(r-\delta)t+X_t}\mid \Q\big)={\cal
L}\big(e^{-(\delta-r)t-X_t}\mid\tilde{\Q}\big),
\end{equation}

meaning equality in law. Otherwise we call trhe Lévy market
\emph{Assymetric}. In view of \eqref{e:tripletilde}, and due to
the fact that the characteristic triplet determines the law of a
L\'evy processes, we obtain that a necessary and sufficient
condition for \eqref{e:symmetry} to hold is

\begin{equation}\label{e:pisymmetric}
\Pi(dy)=e^{-y}\Pi(-dy).
\end{equation}
This ensures $\tilde{\Pi}=\Pi$, and from this follows
$a-(r-\delta)=\tilde{a}-(\delta-r)$, giving \eqref{e:symmetry}, as
we always have $\tilde{\sigma}=\sigma$. As pointed out by
\citeasnoun{FajardoMordecki2006b} condition \eqref{e:pisymmetric}
answers a question raised \citeasnoun{CarrChesney96}. With this
condition in mind we can obtain the following result.
\begin{corollary}[Bates' $x$ \% Rule]
 Take $r=\delta$ and
assume (\ref{e:pisymmetric}) holds, we have
\begin{equation}\label{e:duality4}
c(F_0,K_c,r,\tau,\psi)=(1+x) \;p(F_0,K_p,r,\tau,\psi),
\end{equation}
where $K_c=(1+x)F_0$ and $K_p=F_0/(1+x)$, with $x>0$.
\end{corollary}

\begin{proof}
Follows directly from Lemma 2.1. Since $r=\delta$ and
$\psi=\tilde{\psi}$.
\end{proof}
From here calls and puts at-the-money ($x=0$) should have the same
price. As we mention this $x\%-$rule, in the context of Merton's
model was obtained by \citeasnoun{Bates97}. That is, if the call
and put options have strike prices $x\%$ out-of-the money relative
to the forward price, then the call should be priced $x\%$ higher
than
the put.\\

\subsection{Empirical Evidence of Symmetry}
In \citeasnoun{FajardoMordecki2006b} several concrete models
proposed in the literature are reviewed. More exactly, L\'evy
markets with jump measure of the form

\begin{equation}\label{e:beta}
\Pi(dy)=e^{\beta y}\Pi_0(dy),
\end{equation}
where $\Pi_0(dy)$ is a symmetric measure, i.e.
$\Pi_0(dy)=\Pi_0(-dy)$, everything with respect to the risk
neutral measure $\Q$.\\



As a consequence of \eqref{e:pisymmetric},
\citeasnoun{FajardoMordecki2006b} found that the market is
symmetric if and only if $\beta=-1/2$. Then, as we have seen when
the market is symmetric, the skewness premium
is obtained using the $x\%-$rule. \\

Although from the theoretical point of view the assumption
\eqref{e:beta} is a real restriction, most models in practice
share this property, and furthermore, they have a jump measure
that has a Radon-Nikodym density. In this case, we have
\begin{equation}\label{e:betadensity}
\Pi(dy)=e^{\beta y}p(y)dy,
\end{equation}
where $p(y)=p(-y)$, i.e. the function $p(y)$ is even. More
precisely, all parametric models that we found in the literature,
in what concerns L\'evy markets, including diffusions with jumps,
can be re\-pa\-ra\-me\-tri\-zed in the form \eqref{e:betadensity}:
The Generalized Hyperbolic model proposed by
\citeasnoun{EberleinPrause2000}, The Meixner model proposed by
\citeasnoun{Scoutens2001} and The CGMY model proposed by
\citeasnoun{CGMY02}. Recently, \citeasnoun{Fajardomordecki2007} shows that under some conditions the Time Changed Brownian motion with drift is also included in this class. Then, they show
that the resulting processes will satisfy the above symmetry if and only if the drift equal -1/2.\\

Using the risk neutral market measure and the Esscher transform
measure as EMM, \citeasnoun{FajardoMordecki2006b} obtain evidence
that empirical risk-neutral markets are not symmetric. Then, the
question naturally arises: How to obtain a Put-call symmetry,
under absence of symmetry? In what follows we try to answer this question. \\

Henceforth take $r=\delta$. We need the following assumption
\begin{assumption}
Option prices are monotonic with respect to the symmetry parameter
$\beta$.
\end{assumption}
Our main result is stated as follows.
\begin{theorem}{\label{bergtum}}
Consider Lévy measures given by (\ref{e:beta}). Under Assumption
1, if $\beta\gtrless-1/2$ then
\begin{equation}\label{e:duality41}
C(F_0,K_c,r,\tau,\psi)\gtrless(1+x) \;P(F_0,K_p,r,\tau,\psi),
\end{equation}
where $K_c=(1+x)F_0$ and $K_p=F_0/(1+x)$, with $x>0$.
\end{theorem}
\begin{proof}

 We have that
$$\beta\gtrless-1/2\Longleftrightarrow
\beta\gtrless\widetilde{\beta}:=-\beta-1.$$ Then,
$\Pi(dy)=e^{\beta y}\Pi_0(dy)$ has
$\beta\gtrless\widetilde{\beta}$ of $\widetilde{\Pi}=e^{-(1+\beta)
y}\Pi_0(dy)$. By monotonicity
\begin{eqnarray*}
  C(F_0,K_c,r,\tau,a,\sigma,\Pi) &\gtrless& C(F_0,K_c,r,\tau,a,\sigma,\widetilde{\Pi}) \\
   &=& (1+x)P(F_0,K_c,r,\tau,a,\sigma,\Pi),
\end{eqnarray*}

were the last equality is obtained from duality and the fact that $\widetilde{\widetilde{\Pi}}=\Pi$.\\

The same can be obtained if put prices were decreasing on $\beta$,
we have: $\beta\gtrless-1/2$ implies
\begin{eqnarray*}
 (1+x)P(F_0,K_c,r,\tau,a,\sigma,\Pi) &\lessgtr& (1+x)P(F_0,K_c,r,\tau,a,\sigma,\widetilde{\Pi})\\
   &=& C(F_0,K_c,r,\tau,a,\sigma,\Pi),\;\forall x>0,
\end{eqnarray*}

\end{proof}
\begin{remark}
In the particular case of the GH distributions Assumption 1, can
be guaranteed by Th. 4.2 in \citeasnoun{Bergenthum07}.
\end{remark}
\subsection{Diffusions with jumps}
Consider the jump - diffusion model proposed by
\citeasnoun{Merton76}. The driving L\'evy process in this model
has L\'evy measure given by

\begin{equation*}
\Pi(dy)=\lambda\frac
1{\delta\sqrt{2\pi}}e^{-(y-\mu)^2/(2\delta^2)}dy,
\end{equation*}
and is direct to verify that condition \eqref{e:pisymmetric} holds
if and only if $2\mu+\delta^2=0$. This result was obtained by
\citeasnoun{Bates97} for future options, that result is obtained
as a
particular case.\\

Note that in that model $\beta=\frac {\mu}{\delta^2}$, so we
obtain that sufficient conditions can be replaced by $
\mu+\delta^2/2\gtrless 0 $, as also \citeasnoun{Bates97} found.

\section {Skewness Premium}
In order to study the sign of SK, lets analyze the following data on
S\&P500 American options in 08/31/2006 that matures in 09/15/2006
with future price $F=1303.82$. To verify if the Bates' rule holds we
need to interpolate some non-observed option prices. To this end we
use a cubic spline, as we can see in
Fig. \ref{fig:cpsp500}.\\

\begin{figure}[!htb]\label{fig:cpsp500}
  \begin{minipage}[b]{0.5\linewidth}
    \begin{center}
        \includegraphics[width=\linewidth]{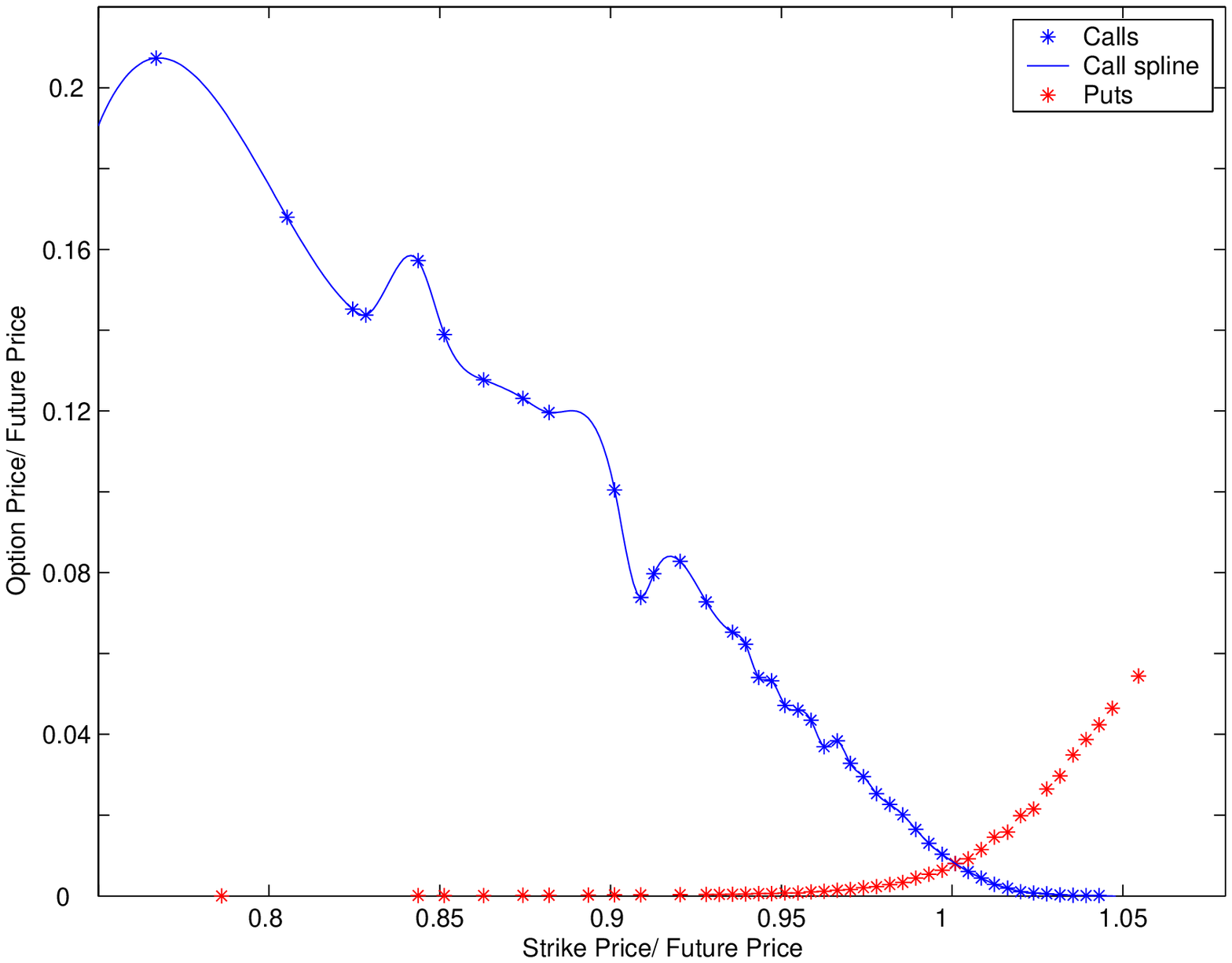}
    \end{center}
  \end{minipage}
  \begin{minipage}[b]{0.5\linewidth}
    \begin{center}
        \includegraphics[width=\linewidth]{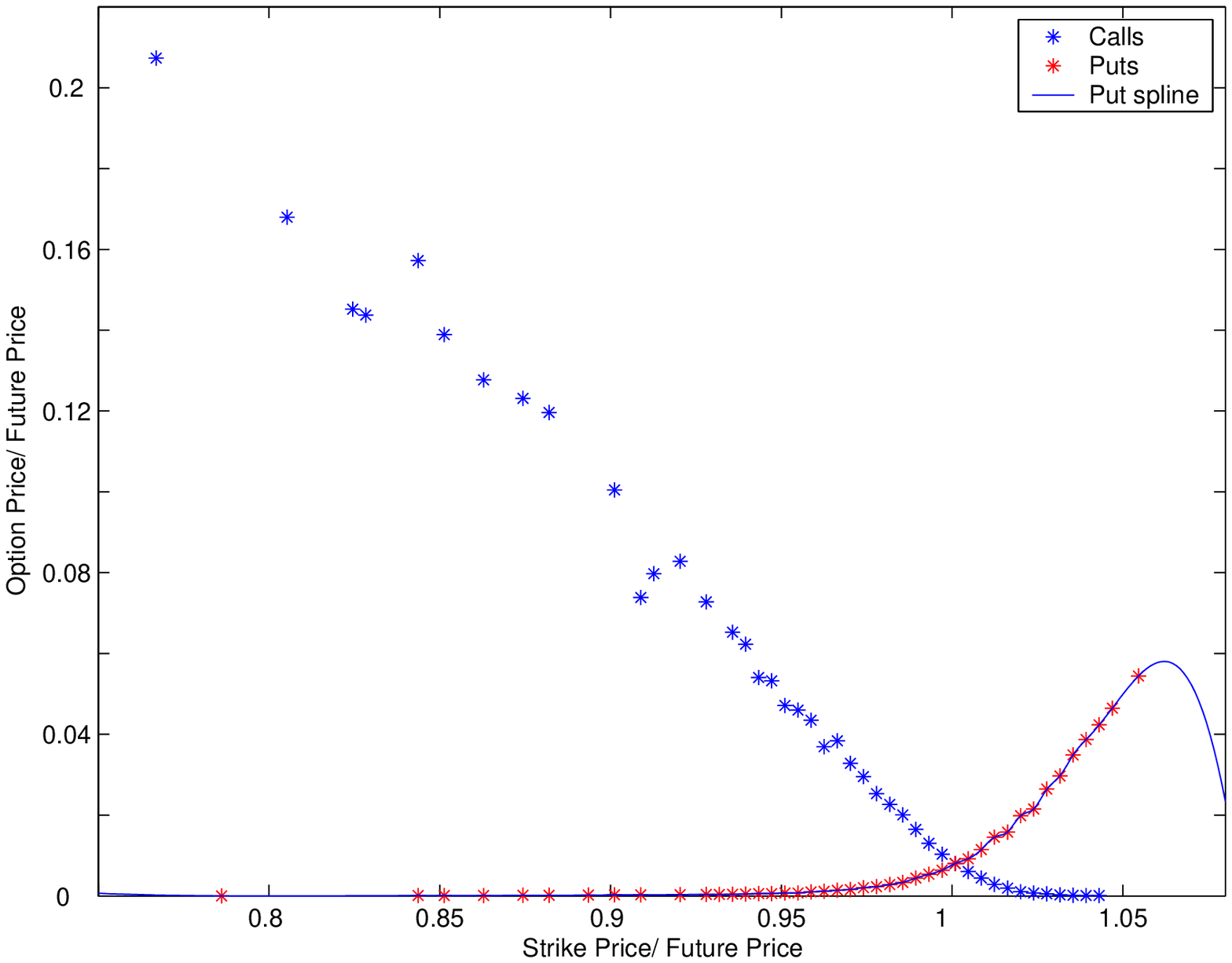}
    \end{center}
  \end{minipage} \hfill
   \caption{ Observed Call and Put prices on S\&P500 in 08/31/2006}
        \label{fig:cpsp500}
\end{figure}
\medskip

 The $x\%$ Skewness Premium is defined as the
percentage deviation of $x\%$ OTM call prices from $x\%$ OTM put
prices. The interpolating calls and put prices for the
non-observed strikes are presented in Tables 1 and 2 at the end.
We can see in both tables that this rule does not hold. Moreover,
for OTM options usually $x_{obs}<x$, what implies $\frac
{c}{p}-1<x$ and for ITM options, $x_{obs}>x$,
implying $\frac {c}{p}-1>x$.\\

\begin{table}[ht]
\begin{center}
\label{otmintput}
 \begin{tabular}{|r|r|r|r|r|} \hline
 $K_c$ & $K_p=F^2/K_c$ & $x=K_c/F-1$ &   $x_{obs}=c_{obs}/p_{int}-1$ & $ x-x_{obs}$  \\
\hline
1230 &    1382.07 &   -0.05662 &   0.050681 &    -0.1073 \\

      1235 &   1376.475 &   -0.05278 &    0.13642 &    -0.1892 \\

      1240 &   1370.925 &   -0.04895 &   0.115006 &   -0.16395 \\

      1245 &   1365.419 &   -0.04511 &   0.197696 &   -0.24281 \\

      1250 &   1359.957 &   -0.04128 &   0.277944 &   -0.31922 \\

      1255 &   1354.539 &   -0.03744 &   0.280729 &   -0.31817 \\

      1260 &   1349.164 &   -0.03361 &   0.536286 &    -0.5699 \\

      1265 &   1343.831 &   -0.02977 &   0.574983 &   -0.60476 \\

      1270 &   1338.541 &   -0.02594 &   0.606719 &   -0.63266 \\

      1275 &   1333.291 &    -0.0221 &   0.675372 &   -0.69748 \\

      1280 &   1328.083 &   -0.01827 &   0.691325 &   -0.70959 \\

      1285 &   1322.916 &   -0.01443 &   0.966306 &   -0.98074 \\

      1290 &   1317.788 &    -0.0106 &   0.904839 &   -0.91544 \\

      1295 &     1312.7 &   -0.00676 &   0.794059 &   -0.80082 \\

      1300 &   1307.651 &   -0.00293 &    0.78018 &   -0.78311 \\
     \hline
      1305 &   1302.641 &   0.000905 &   0.614561 &   -0.61366 \\

      1310 &   1297.669 &    0.00474 &   0.532798 &   -0.52806 \\

      1315 &   1292.735 &   0.008575 &   0.427299 &   -0.41872 \\

      1320 &   1287.838 &    0.01241 &   0.108911 &    -0.0965 \\

      1325 &   1282.979 &   0.016245 &   -0.11658 &   0.132826 \\

      1330 &   1278.155 &   0.020079 &   -0.45097 &   0.471053 \\

      1335 &   1273.368 &   0.023914 &   -0.50378 &   0.527697 \\

      1340 &   1268.617 &   0.027749 &   -0.61306 &   0.640807 \\

      1345 &   1263.901 &   0.031584 &   -0.73872 &   0.770305 \\

      1350 &    1259.22 &   0.035419 &   -0.81448 &   0.849896 \\

      1355 &   1254.573 &   0.039254 &   -0.80297 &   0.842224 \\

      1360 &   1249.961 &   0.043089 &   -0.82437 &   0.867454 \\

\hline
\end{tabular}
\caption{Options prices Interpolating Put prices}
\end{center}
\end{table}

\begin{table}[ht]
\begin{center}
\label{itmintcall}
\begin{tabular}{|r|r|r|r|r|}
\hline

 $K_p$ & $K_c=F^2/K_p$ & $x=F/K_p-1$ &   $x_{obs}=c_{int}/p_{obs}-1$ & $ x-x_{obs}$  \\
\hline

        1250 &   1359.957 &   0.043056 &   -0.88837 &   0.931421 \\

      1255 &   1354.539 &     0.0389 &   -0.86897 &   0.907873 \\

      1260 &   1349.164 &   0.034778 &   -0.85655 &   0.891331 \\

      1265 &   1343.831 &   0.030688 &   -0.78107 &    0.81176 \\

      1270 &   1338.541 &    0.02663 &   -0.70531 &   0.731941 \\

      1275 &   1333.291 &   0.022604 &   -0.63926 &   0.661869 \\

      1280 &   1328.083 &   0.018609 &   -0.51726 &   0.535865 \\

      1285 &   1322.916 &   0.014646 &   -0.31216 &   0.326801 \\

      1290 &   1317.788 &   0.010713 &   -0.20329 &   0.214005 \\

      1295 &     1312.7 &   0.006811 &   -0.03659 &   0.043397 \\

      1300 &   1307.651 &   0.002938 &   0.090739 &    -0.0878 \\
     \hline
      1305 &   1302.641 &    -0.0009 &   0.130843 &   -0.13175 \\

     1310 &   1297.669 &   -0.00472 &   0.252541 &   -0.25726 \\

      1315 &   1292.735 &    -0.0085 &   0.261905 &   -0.27041 \\

      1320 &   1287.838 &   -0.01226 &   0.242817 &   -0.25507 \\

      1325 &   1282.979 &   -0.01598 &   0.346419 &    -0.3624 \\

      1330 &   1278.155 &   -0.01968 &   0.183207 &   -0.20289 \\

      1335 &   1273.368 &   -0.02336 &   0.237999 &   -0.26135 \\

      1340 &   1268.617 &     -0.027 &   0.145858 &   -0.17286 \\

      1345 &   1263.901 &   -0.03062 &   0.152637 &   -0.18325 \\

      1350 &    1259.22 &   -0.03421 &   0.101211 &   -0.13542 \\

      1355 &   1254.573 &   -0.03777 &   -0.03964 &   0.001869 \\

      1360 &   1249.961 &   -0.04131 &   0.028337 &   -0.06965 \\

      1365 &   1245.382 &   -0.04482 &    -0.0101 &   -0.03472 \\

      1375 &   1236.325 &   -0.05177 &    -0.0451 &   -0.00667 \\
  \hline

\end{tabular}

\caption{Options prices Interpolating Call prices}
\end{center}
\end{table}

Then we want to know for what distributional parameter values we
can capture the observed vies in these option price ratios. To
this end we use the following definition introduced by
\citeasnoun{Bates91}.

\begin{equation}\label{sk} SK(x)=\frac
{c(S,T;X_c)}{p(S,T;X_p)}-1,\;\;\hbox{for European
Options,}\end{equation}
$$SK(x)=\frac {C(S,T;X_c)}{P(S,T;X_p)}-1,\;\;\hbox{for American
Options,}$$ where $X_p=\frac {F}{(1+x)}<F<F(1+x),\;\;x>0$.\\

The SK was addressed for the following stochastic processes:
Constant Elasticity of Variance (CEV), include arithmetic and
geometric Brownian motion. Stochastic Volatility processes, the
benchmark model  being those for which volatility evolves
independently of the asset price. And the Jump-diffusion
processes, the benchmark model is the Merton's (1976) model. For
that classes \citeasnoun{Bates96} obtained the following result.
\begin{proposition}[\citeasnoun{Bates96}]
For European options in general and for American options on
futures, the SK has the following  properties for the above
distributions.
\begin{itemize}
 \item[i)] $ SK(x)\lessgtr x$ for CEV processes with
$\rho\lessgtr 1$.
 \item[ii)] $SK(x)\lessgtr x$ for jump-diffusions with log-normal jumps depending on whether $2\mu+\delta^2\lessgtr
 0$.
  \item[iii)] $SK(x)\lessgtr x$ for Stochastic Volatility processes depending on whether $\rho_{S\sigma}\lessgtr
  0$.
\end{itemize}
 Now in equation (\ref{sk}) consider
$$X_p=F(1-x)<F<F(1+x),\;\;x>0.$$ Then,
\begin{itemize}
\item[iv)] $ SK(x)< 0$ for CEV processes only if $\rho<0$.
\item[v)] $ SK(x)\geq 0$ for CEV processes only if $\rho\geq 0$.
\end{itemize}
When $x$ is small, the two SK measures will be approx. equal.
 For in-the-money options $(x<0)$, the propositions are reversed. Calls $x\%$ in-the-money should cost
$0\%-x\%$ less than puts $x\%$ in-the-money.
\end{proposition}
\begin{proof}
See \citeasnoun{Bates91}.
\end{proof}
Now using Th. \ref{bergtum}, we can extend Bates' result to Lévy
processes.
\begin{proposition}
Consider Lévy measures given by (\ref{e:beta}). Under Assumption
1, we have
  $$\beta\gtrless-1/2 \Rightarrow SK(x)\gtrless x ,\;\; x>0.$$
\end{proposition}
And from here we obtain the sign of the excess of skewness premium
for a huge class of Lévy processes.

 \section{Monotonicity and
Symmetry Parameter} As we have seen in the last section we need
the monotonicity of option prices with respect to the symmetry
parameter to obtain our main result. The literature had study
extensively the monotonicity properties of option prices. The main
idea is to exploit the {\it convexity preserving
property}\footnote{We say that a model is {\it convexity
preserving}, if for any convex contract function, the
corresponding price is convex as a function of the price  of the
underlying asset at all times prior to maturity. Many models do
not satisfy this property as for example general stochastic
volatility models.}, to obtain the monotonicity of option prices
with respect to certain
 parameter of the model. See \citeasnoun{BGW96}, \citeasnoun{ElkarouiJS98} and \citeasnoun{EkstromTysk05}. \\

 On the other hand, this question is very related to the
 ordering of option prices by changing the equivalent martingale measure. That is, imposing conditions on
 the predictable characteristic of the underlying process, an ordering in option prices with respect to the equivalent
 martingale measures is established, see \citeasnoun{Bellamy00}, \citeasnoun{HendersonHobson03}, \citeasnoun{Henderson05}, \citeasnoun{Jakubenas2002}, \citeasnoun{GuschinMordecki02} and \citeasnoun{Bergenthum06}. \\

 But we are interested in the possible
mispecification in the models when using a fixed equivalent
martingale measure. That is, if we change the parameter $\beta$ on
the Lévy measure described by (\ref{e:beta}) what happen with the
option price. To answer partially that question, we can apply
Lemma 5.1 in \citeasnoun{GuschinMordecki02} for a certain group of
Lévy measures (\ref{e:beta}), that is if the Lévy measure  satisfy
the assumptions of that lemma, we obtain and order in option
prices. Then, Assumption 1 will be satisfied.
\section{Conclusions}
Under a given risk neutral probability measure. We use a measure
of \emph{symmetry}, introduced by
\citeasnoun{FajardoMordecki2006b}, to address the Skewness premium
under absence of symmetry. First, we analyze the sign of the
Skewness premium using data
 from S\&P500  and we obtain evidence that Bates' $x\%$ rule does not hold. In that
 case we derive sufficient conditions for excess of SK to be positive or negative. In particular on the
 symmetry parameter. In this way we obtain simply diagnostic to
 observe what L\'evy model deals with both the behavior of the
 underlying and with the sign of SK.\\

 Interesting issues to study in future works are the empirical evidence of Assumption 1 and under what conditions, on the symmetry parameter, the monotonicity of option prices with respect to
  symmetry parameter holds. In that sense the results obtained by
  \citeasnoun{Bergenthum07} for the GH distributions can bring some insights.
\bibliographystyle{econometrica}
\bibliography{skewness241008}
\end{document}